\documentclass[12pt,a4paper]{amsart}
\usepackage{preamble}

\title
[Adjoint shadowing lemma] 
{
Backpropagation in hyperbolic chaos via adjoint shadowing 
}

\begin{document}

\begin{abstract}

To generalize the backpropagation method to both discrete-time and continuous-time hyperbolic chaos, we introduce the adjoint shadowing operator $\mathcal{S}$ acting on covector fields. We show that $\mathcal{S}$ can be equivalently defined as:
\begin{enumerate}
\item $\mathcal{S}$ is the adjoint of the linear shadowing operator $S$;
\item $\mathcal{S}$ is given by a `split then propagate' expansion formula;
\item $\mathcal{S}(\om)$ is the only bounded inhomogeneous adjoint solution of $\om$.
\end{enumerate}
By (a), $\mathcal{S}$ adjointly expresses the shadowing contribution, a significant part of the linear response, where the linear response is the derivative of the long-time statistics with respect to system parameters.
By (b), $\mathcal{S}$ also expresses the other part of the linear response, the unstable contribution.
By (c), $\mathcal{S}$ can be efficiently computed by the nonintrusive shadowing algorithm in \cite{Ni_nilsas}, which is similar to the conventional backpropagation algorithm.
For continuous-time cases, we additionally show that the linear response admits a well-defined decomposition into shadowing and unstable contributions.

\end{abstract}

\maketitle

\section{Introduction}

\subsection{Literature review}
\hfill\vspace{0.1in}

The adjoint method is a method for efficiently computing the gradient of an objective (or observable) function $\Phi$ with respect to many system parameters $\gamma$; 
it typically involves computing some adjoint operators, and it is typically evaluated on one or a few orbits of a dynamical system.
Its cost is almost independent of the number of parameters.
Conventionally, the adjoint method solves one inhomogeneous adjoint equation, which is the pullback of covectors, and it runs backwards in time.
Hence, adjoint method is also called the backpropagation method.

The adjoint method is the adjoint of the pathwise perturbation method.
Hence, the conventional adjoint method works only for stable systems or short-time.
For unstable systems with positive Lyapunov exponents, we run into the so-called `gradient explosion' phenomenon, that is, the adjoint solutions (or the pullback of covectors) 
grow exponentially fast as we increase the orbit length.
As a result, the derivative of the objective at a specific time is poorly defined when the orbit is long, and we need to consider the perturbation  of an averaged objective over a long-time or a measure.
The derivative of averaged objective is called the linear response, and conventional adjoint method does not work on this case.
Pioneering attempts to fix the adjoint method, such as the gradient clipping, do not have good math explanations and not always work well \cite{clip_gradients2}.

The linear response is theoretically studied in hyperbolic systems, where all Lyapunov exponents are non-zero \cite{Ruelle_diff_maps,Dolgopyat2004,Baladi2007,Gallavotti1996, Gouezel2006,Gouezel2008,Baladi2017}.
The chaotic hypothesis of Gallavotti and Cohen states that many high-dimensional physical systems are hyperbolic \cite{gallavotti_chaotic_hypothesis_1995,Bonetto06,Wormell2019}.
It seems that this hypothesis does not hold strictly, since Wormell and Gottwald recently gave a counter example \cite{Wormell2019,wormell22}.
Also, weather systems seem to have a significant amount of non-hyperbolic directions \cite{LE_weather}, but still seem to have linear responses.
However, even with these counter examples, the morale of the chaotic hypothesis seems to be still valid: real physical systems may have a large ratio of hyperbolic directions or a large hyperbolic region, so theories based on hyperbolicity should still be a very important part of a full solution.

The most well-known linear response formulas are the pathwise perturbation formula, the divergence formula, and the kernel differentiation formula,  all having adjoint versions.
The kernel differentiation formula, also known as the likelihood ratio method in probability context \cite{Rubinstein1989,Reiman1989,Glynn1990,Hairer2010,np}, requires that the system must be stochastic, and the cost is large when the noise is small.
The other two formulas work for both deterministic and stochastic systems; they are not hindered by small noise, but both have their own unique shortcomings.
In this paper, we consider only the deterministic case with hyperbolicity: this was the path for classical dynamical systems theory.

The pathwise perturbation formula averages conventional adjoint formula over a lot of orbits, which is also known as the ensemble method or stochastic gradient method; however, it is cursed by the gradient explosion \cite{Lea2000,eyink2004ruelle,lucarini_linear_response_climate,lucarini_linear_response_climate2}. 
The divergence formula, also known as the transfer operator formula computes the perturbation operator, which is not pointwisely defined, and we have to partition the full phase space to obtain some mollified approximation: this is cursed by dimensionality \cite{Keane1998,Froyland2007,Liverani2001,Ding1994,Pollicott2000,Galatolo2014,Galatolo2014a,Wormell2019,Crimmins2020,Antown2022,Wormell2019a,Froyland2013a,SantosGutierrez2020,Bahsoun2018,Zhang2020}.
A promising direction is to blend the two formulas, but the obstruction was that the split lacks smoothness \cite{abramov2007blended}.
There are pointwisely defined formulas such as \cite{Ruelle_diff_maps,Gouezel2008}, but they still involve terms whose expressions are not obvious; moreover, they are not recursive, so can not be computed efficiently; even if we use our tools to make those formulas recursive, the number of recursive relations would be large.

Our recent work, the fast response formula and algorithm, solve the smoothness issue for the blended approach, and computes the linear response of discrete-time hyperbolic systems on a sample orbit \cite{fr}.
In some sense the formula is the ergodic theorem for linear response: it computes the pushforward of $2u+2$ $M$-dimensional vectors 
on an orbit, and then the average of some pointwise functions of these vectors converge to the linear response.
Here $M$ is the phase space dimension, and $u$ is the number of unstable Lyapunov exponents.
This number of recursive relations seems to be close to the least possible, since we need at least $u$ many relations to capture all the unstable perturbative behaviors.

In the discrete-time fast response algorithm, the linear response is decomposed into shadowing and unstable contributions.
The shadowing contribution ($SC$) is given by the shadowing vector, which is the difference between two orbits close to each other but with perturbed parameters \cite{Bowen_shadowing,Pilyugin_shadow_linear_formula}.
Due to the similarity between this characterization and the conventional pathwise perturbation method, the shadowing vector can be efficiently computed by the nonintrusive shadowing algorithm \cite{Ni_NILSS_JCP,Ni_fdNILSS,Ni_CLV_cylinder}.
When the ratio of unstable directions is low, with some additional statistical assumptions, shadowing can sometimes be a good approximation of the entire linear response \cite{Ruesha}.
Nonintrusive shadowing is also necessary for efficiently computing the unstable contribution in the fast response algorithm, which is given by a tangent version of the equivariant divergence formula \cite{fr}.

Our works above are not `adjoint', and their costs are linear to the number of parameters.
Adjoint theories and algorithms are nontrivial: we need more work finding good characterizations than just transposing matrices.
Finding a neat characterization is especially difficult in continuous-time, where the flow direction is a Lyapunov vector with zero exponent and can be easily captured.
Hence, we need to design special treatment for the flow direction, which also assembles well with other directions.

Moreover, there are some lingering questions about shadowing methods for continuous-time.
In particular, it was not clear that, in continuous-time cases, different previous shadowing algorithms were computing the same quantity.
It was also not clear whether shadowing methods compute a significant part of the linear response.
We shall give positive answers to these questions.

\subsection{Main results}
\hfill\vspace{0.1in}
\label{s:mainresu}

As part of the task to develop adjoint methods for the parameter-gradient of long-time statistics of hyperbolic chaos, this paper develops the adjoint shadowing theory for both discrete- and continuous-time systems.
Shadowing is of particular interest because: 
\begin{itemize}
  \item 
  The shadowing contribution $SC$ can sometimes be a good approximation of the entire linear response; 
  \item
  Shadowing is also necessary for computing the other part of linear response, the unstable contribution $UC$;
  \item 
  The shadowing covector looks very similar to the conventional backpropagation solution, so it can be implemented relatively easily.
\end{itemize}
We shall also answer some lingering questions for shadowing methods in continuous-time.

For the discrete-time system of a map $f$ parameterized by $\gamma$, section~\ref{s:adjoint1} first shows that the adjoint system on the covector space, given by recursively applying $f^*$, 
is also hyperbolic, where the pullback operator $f^*$ is basically the transposed Jacobian matrix.
So we can define  oblique projections onto the unstable and stable cotangent subspace, denoted by $\cP^u$ and $\cP^s$.

\Cref{s:pathwiseasl} and \cref{s:december} prove \cref{l:pathasl} and \cref{t:AS}, the equivalence of three characterizations of the adjoint shadowing operator $\cS$ for discrete-time.
\Cref{l:pathasl} is pathwise, whereas \cref{t:AS} is stated on an attractor with the physical measure.
First, fix a $C^\infty$ observable (also called objective) function $\Phi$.
For an orbit on the attractor $K$, the long-time-average statistic typically coincides with the SRB measure $\rho$ supported on the attractor.
Let $\rho(\Phi)$ denote the integration of $\Phi$ according to $\rho$.
The linear response is $\delta\rho(\Phi)$, where 
\[ \begin{split}
\left. \delta(\cdot):=\pp{ (\cdot)}{\gamma} \right|_ {\gamma=0},
\end{split} \]
$\gamma$ is some parameter of the system, 
Let $\cX(K)$, $\cX^{\alpha}(K)$, $\cX^{*}(K)$, and $ \cX^{*\alpha}(K)$ be the spaces of continuous or Holder-continuous vector and covector fields on $K$.
Let $S$ be the linear shadowing operator to be defined in section~\ref{s:hypersha}.
Denote $(\cdot)_n:=(\cdot)(f^nx)$.
Also, for any vector $X$ and covector $\om$, we denote the product $X\om:=\om X:=\om(X)$.

\begin{restatable}[adjoint shadowing lemma for discrete-time]{theorem}{goldbach}
\label{t:AS}
On a compact mixing axiom A attractor with physical measure $\rho$,
the adjoint shadowing operator $\cS:\cX^{*\alpha}(K) \rightarrow \cX^* (K)$ is equivalently defined by the following characterizations:
  \begin{enumerate}
  \item 
  $\cS$ is the linear operator such that
  \[ \begin{split}
    \rho(\om S(X)) = \rho(\cS(\om) X)
    \quad \textnormal{for any } X\in\cX^\alpha(K).
  \end{split} \]
  Hence, if $X=\delta f\circ f^{-1}$,
  $\nu=\cS(d\Phi)$, then the shadowing contribution is
  \[ \begin{split}
    SC = \rho(d\Phi S(X)) = \rho(\cS(d\Phi) X)
    =  \lim_{N\rightarrow \infty} \frac 1{N} \sum_{n=1}^N  \nu_n X_n \,.
  \end{split} \]
  Here the last equality holds for almost all initial conditions on the basin of the attractor according to the Lebesgue measure.
  \item 
  $\cS(\om)$ has the expansion formula given by a `split-propagate' scheme,
  \begin{equation*}
    \cS ( \om ) 
    = \sum_{n\ge 0} f^{*n} \cP^s \om_n
    -  \sum_{n\le -1}  f^{*n} \cP^u \om_n\,.
  \end{equation*}
  \item 
  The shadowing covector $\nu=\cS(\om)$ is the unique solution of the inhomogeneous adjoint equation,
  \[ \begin{split}
  \nu = f^* \nu_1 + \om, 
  \quad \textnormal{where} \quad \nu_1:=\nu \circ f.
  \end{split} \]
  \end{enumerate}
Moreover, $\cS$ preserves Holder continuity.
\end{restatable}

For the continuous-time systems, the shadowing contribution was not clearly defined before.
To do this, first, \cref{s:S} defines the quotient space
\[ \begin{split}
  A(K):=\{ (v,\eta): v\in \cX^\alpha (K), \nabla_F v\in  \cX^\alpha (K), \eta\in C^\alpha(K) \} \, / \sim.
\end{split} \]
where the equivalent relation $\sim$ is defined as
\[ \begin{split}
(v_1,\eta_1)\sim (v_2,\eta_2)
\quad \textnormal{iff} \quad 
  \cL_F v_1 +\eta_1 F = \cL_F v_2 +\eta_2 F.
\end{split} \]
We define the (tangent) shadowing operator $S:\cX^\alpha(K)\rightarrow A(K)$ as the 
\[ \begin{split}
S: X \mapsto [v, \eta],
\quad \textnormal{where} \quad 
\cL_F v + \eta F
= X .
\end{split} \]
Here $[v,\eta]$ is the equivalent class of $(v,\eta)$ according to $\sim$.
Let $F(\psi)$ be the derivative of $\psi$ along $F$,
define the space of pairs of a covector field and a scalar function, 
\[ \begin{split}
  \cA(K):=\{(\om, \psi) \,|\, 
  \om\in \cX^{*\alpha}, 
  \psi\in C^\alpha, 
  F(\psi) \in C^\alpha(K), 
  F(\psi) = \om(F)\}.
\end{split} \]
\Cref{s:sleep} shows that, for $(\om, \psi)\in\cA(K)$, $\llangle S(X);\om,\psi\rrangle$ is well-defined, where
\[ \begin{split}
  \llangle v,\eta;\om,\psi\rrangle 
  := \rho (\om v) - \rho(\eta \psi).
\end{split} \]
Hence we can well-define the shadowing contribution $SC$ as
\[ \begin{split}
  SC :=\llangle S(X); d\Phi, \Phi -\rho(\Phi) \rrangle,
\end{split}\]
These definitions can all be made pathwise.
\Cref{s:waiyi} shows that the other part of the linear response, the unstable contribution $UC$, is indeed only in the unstable direction.
Hence, there is a chance that $UC$ is small when the unstable dimensional is low, using arguments similar to \cite{Ruesha}.

\Cref{s:pathaslF} proves \cref{l:pathaslF}, the pathwise adjoint shadowing lemma in continuous-time;
\cref{s:aslF} proves its attractor version, \cref{t:ASF}.
First, we define some notations. 
Let $f$ be the flow of the vector field $F+\gamma X$, where $F$ and $X$ are two fixed vector fields.
Denote $(\cdot)_t:=(\cdot)(f^tx)$.
Let $\eps^c$ be the covector in the center subspace such that $\eps^c (F)=1$.
Let $\cL_{(\cdot)}(\cdot)$ be the Lie derivative.
Let $\nabla_{(\cdot)}(\cdot)$ be a Riemannian covariant derivative; so in $\R^M$, $\nabla_F$ is $\partial/\partial t$. 
Define the derivative along a covector, say $\nabla_\nu F$, as $\nabla_F\nu-\cL_F\nu$; by \cref{e:hey} in \cref{s:adjhyperF}, $\nabla_\nu F = -(\nabla F)^T \nu$ when $\cM=\R^M$.

\begin{restatable}[adjoint shadowing lemma for continuous-time]{theorem}{silverbach}
\label{t:ASF}
On a compact mixing axiom A attractor with physical measure $\rho$ and decay of correlations in \cref{e:deep,e:published}, 
the adjoint shadowing operator $\cS: \cA(K) \rightarrow \cX^{*}(K)$ is equivalently defined by the following characterizations:
  \begin{enumerate}
  \item 
  $\cS$ is the linear operator such that
  \[ \begin{split}
  \llangle S(X); \om, \psi \rrangle =\rho(X \cS(\om, \psi))
  \quad \textnormal{for any }
  X\in \cX^\alpha,
  \end{split} \]
  Hence, if $X=\delta f$,
  $\nu =\cS (d\Phi, \Phi -\rho(\Phi))$, 
  then the shadowing contribution is
\[ \begin{split}
  SC =\llangle S(X); d\Phi, \Phi -\rho(\Phi) \rrangle
  =\rho(X \cS(d\Phi, \Phi-\rho(\Phi))) 
  = \lim_{T\rightarrow \infty} 
  \frac 1{T} \int_{0}^T  \nu_t X_t dt \,.
\end{split} \]
  \item 
  $\cS$ has the `split-propagate' expansion formula
\begin{equation*}
  \cS (\om,\psi ) 
  = \int_{t\ge 0} f^{*t} \om^s_t dt
  -  \int_{t\le 0} f^{*t} \om^u_t dt
  - \psi \eps^c.
\end{equation*}
  \item 
  The shadowing covector 
  $\nu = \cS(\om,\psi)$ is the unique solution of the inhomogeneous adjoint ODE,
  \[ \begin{split}
  \nabla_F \nu - \nabla_\nu F 
  = \cL_F\nu
  = - \omega
  \; \textnormal{ on  } K,
  \quad \textnormal{} \quad 
  \nu F (x)  = - \psi (x)
  \; \textnormal{ at all or any } x\in K.
  \end{split} \]
  \end{enumerate}
Moreover, $\cS$ preserves Holder continuity.
\end{restatable}

\Cref{s:appdiscrete} and \cref{s:appaslf} explain and discuss the utilities of each characterization in the adjoint shadowing lemma, for both discrete- and continuous-time.
First, the expression of $SC$ in (a) by $\cS$ gives an `adjoint' method for $SC$, which means that the main term $\nu$ can be computed by a few recursive relations on an orbit, and does not depend on the parameter, or $X$:
this typically requires some non-trivial characterizations of some adjoint operators, hence the name.
The integration in $SC$ can also be sampled on an orbit.

The governing equation in (c) is the same as the conventional backpropagation method, but here we added boundedness.
This allows the shadowing covector and hence the shadowing contribution be efficiently computed via the nonintrusive shadowing algorithm in \cite{Ni_nilsas} (also see \cite{Blonigan_2017_adjoint_NILSS}, which is explained at the end of \cref{s:appaslf}).
Nonintrusive means to compute $O(u)$ many vectors recursively on one orbit, where $u$ is the unstable dimension.
The nonintrusive adjoint shadowing algorithm was used on a $10^6$ dimensional fluid problem; the cost was on the same order of simulating the flow, and the shadowing contribution is a useful approximation of the entire linear response.

Somewhat surprisingly, using (b), we can see that the other part of the linear response, the unstable contribution, can also be expressed by $\cS$ applied on an equivariant divergence.
This allows the unstable contribution be also computed with $O(u)$ cost per step.
The adjoint theory and algorithm for the unstable contribution of discrete-time systems are in \cite{TrsfOprt,far}.
The theory of continuous-time case uses adjoint shadowing lemma in continuous-time \cite{vdivF}.

\Cref{s:future} discusses a plausible future line of study: adding randomness.
It is known that many systems do not have linear response due to lacking hyperbolicity.
However, engineers can still ask for approximate derivatives.
The most plausible solution seems to be adding local noise to bad regions with poor hyperbolicity.
Then we can use the kernel differentiation trick to circumvent the bad properties of the dynamics.
However, this potential program still misses a few key techniques, in particular, we need formulas for the random and deterministic regions to communicate.

\section{Preparations for discrete-time}
\label{s:prep1}

\subsection{Hyperbolicity and tangent shadowing}
\hfill\vspace{0.1in}
\label{s:hypersha}

We assume uniform hyperbolicity.
Let $f$ be a smooth diffeomorphism on a smooth Riemannian manifold $\cM$, whose dimension is $M$.
Assume that $K$ is a compact hyperbolic set, that is, $T_K\cM$ (the tangent bundle restricted to K) 
has a continuous $f_*$-invariant splitting into stable and unstable subspaces, $T_KM = V^s \bigoplus V^u$,
such that there are constants $C>0$, $0<\lambda < 1$, and
\[
  \max_{x\in K}|f_* ^{-n}|V^u(x)| ,
  |f_* ^{n}|V^s(x)| \le C\lambda ^{n} \quad \textnormal{for  } n\ge 0.
\]
Here $f_*$ is the pushforward operator on vectors,
when $\cM=\R^M$, $f_*$ is the Jacobian matrix $\partial f/\partial x$.
Define the oblique projection operators $P^u$ and $P^s$, such that
\[ \begin{split}
  v = P^u v + P^s v, \quad \textnormal{where} \quad P^u v\in V^u, P^s v\in V^s.
\end{split} \]
We further assume that $K$ is an attractor, that is, there is an open neighborhood $U$, called the basin of the attractor,
such that $\cap_{n\ge0} f^nU=K$.

Then we review tangent theories and nonintrusive shadowing algorithms.
In this paper, `tangent' means the pushward of vectors, which runs forward in time, and the cost is typically independent of the number of observables $\Phi$.
Adjoint solvers typically involve the pullback of covectors, it runs backward and compute the gradient of one objective with respect to many parameters $\gamma$.

We define the shadowing vector by its characterizing properties.
Take an orbit $\cO:=\{x_n=f(x_{n-1})\}_{n\in \Z}$ on the attractor with discrete topology on all steps, in particular, if $x_1=x_2$, they are still counted as two disjoint points.
Let $\cX(\cO)$ be the space of bounded sequences of vectors.
For any $X\in \cX(\cO)$, or $\{X_n\in T_{x_n} \cM\}_{n\in\Z}$, 
the shadowing vector $v\in \cX(\cO)$ is the only bounded solution of the inhomogeneous tangent equation,
\[ \begin{split}
  v_{n+1} = f_{*} v_n + X_{n+1} ,
  \quad\textnormal{where}\quad
  X:= \delta f \circ f^{-1}.
\end{split} \]
Here the subscript $(\cdot)_n:=(\cdot)(x_n)$.
Since the inhomogeneous tangent equation governs perturbations of an orbit due to $\delta f$, shadowing vector is just the first-order difference between shadowing orbits, which are two orbits with slightly different governing equations, but are always close to each other \cite{Bowen_shadowing}.
Define shadowing operator $S_\cO:\cX(\cO)\rightarrow \cX(\cO)$ on a particular path $\cO$ by
\[ \begin{split}
  S: X\mapsto v.
\end{split} \]
We can further define the linear shadowing operator $S:\cX^\alpha(K) \rightarrow \cX^\alpha(K)$, where $\cX^\alpha(K)$ is the space of Holder-continuous vector fields on $K$.
The appendix~\ref{a:holder} proves that $S$ preserves Holder continuity.
Our definition of $S$ is comparable to characterization (c) of the adjoint operator $\cS$.

We give the expansion formula of the shadowing vector that fits the definition.
First, we write out the conventional inhomogeneous tangent solution with zero initial condition, $v'$.
Propagate previous perturbations to step $k$, then sum up, 
\[ \begin{split}
  v_{k}' = \sum_{n=1}^{k}D^k_{n} X_{n} ,
  \quad
  v_{0}'=0.
\end{split} \]
The expression of the shadowing vector can be obtained similarly, but we should first decompose $X$, then propagate stable components to the future, unstable components to the past.
Both procedures are shown in figure~\ref{f:v'}.
The expansion formula for $v$ is 
\begin{equation} \begin{split} \label{e:expand_v}
  v 
  = S(X)
  = \sum_{n\ge0} f^{n}_{*} P^s X_{-n} 
  - \sum_{n\le-1} f^{n}_{*} P^u X_{-n} ,
\end{split} \end{equation}
This formula works both on an orbit and on the entire attractor.

We call this procedure to obtain $v$ a `split-propagate' scheme; it occurs very frequently in orbit-based linear response theory.
Since we split $X$ into the stable and unstable vectors, their exponential decay makes $v$ uniformly bounded.
The `propagation' is a variant linear superposition law, which guarantees that $v$ is still an inhomogeneous tangent solution: the proof is similar to that of \cref{l:pathasl} in \cref{s:pathwiseasl}.

\begin{figure}[ht] \centering
  \includegraphics[scale=1]{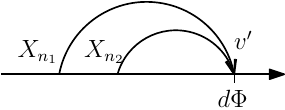}\hfill
  \includegraphics[scale=1]{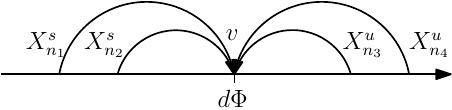}
  \caption{Inhomogeneous tangent solution with zero initial condition (left),
  and the shadowing vector (right).}
  \label{f:v'}
\end{figure}

We compute the shadowing vector by recovering properties in its definition.
Currently, the most efficient algorithm is the nonintrusive shadowing algorithm \cite{Ni_NILSS_JCP,Ni_fdNILSS}.
Nonintrusive means to compute the pushforward or pullback of only $O(u)$ many vectors or covectors; here we compute only $u$ many vectors.
The algorithm solves the nonintrusive shadowing problem on an orbit,
\[ \begin{split} 
  v = v' + \sum_{i=1}^u e^i a_i \,,
  \quad  \mbox{s.t. }
  \ip{v_N, e^i_N}=0
  \quad \textnormal{for all} \quad 
  1\le i\le u.
\end{split}\]
Here $v'$ is a particular inhomogeneous adjoint solution.
Intuitively, the unstable modes are removed by the orthogonal projection at the last step, where the unstable modes are the most significant \cite{Ni_NILSS_JCP}. 
Among all algorithms for stable or shadowing part of the linear response, nonintrusive shadowing is faster than some algorithm, including the stable part of blended response \cite{abramov2007blended,Wang_ODE_LSS,Blonigan_MSS}; and it seems to be faster than the others, where a direct comparison is difficult \cite{BloniganPhdThesis,Lasagna2019,Kantarakias2023}.

\subsection{Physical measure and linear response}
\hfill\vspace{0.1in}

Many hyperbolic sets, say mixing axiom A attractors, admit physical SRB measures, denoted by $\rho$.
That is, fix any $C^\infty$ observable function $\Phi:\cM \rightarrow \R$,
then for almost all $x$ in the attractor basin $U$,
\[ \begin{split}
  \lim_{N\rightarrow \infty} \frac 1N \sum _{k=0} ^{N-1} \Phi (f^kx) 
  = \rho(\Phi)
  : = \int \Phi(x) \rho(dx).
\end{split} \]
In other words, physical measures give the long-time statistic of the chaotic system.
The physical measure has several other characterizations \cite{young2002srb}.

Assume that the system is parameterized by a parameter $\gamma$, the parameterization $\gamma\mapsto f$ is $C^1$ as a map $\R\rightarrow C^r$, where $r\ge2$.
The linear response formula expresses $\delta \rho$ by $\delta f \in C^2$, where
\[ \begin{split}
  \delta(\cdot):=\partial (\cdot)/\partial \gamma
  \quad \textnormal{at } \gamma=0
\end{split} \]
may as well be regarded as small perturbations.
It is proved that \cite{Ruelle_diff_maps}, 
\[
  \delta \rho (\Phi)
  = \lim_{W\rightarrow\infty} \rho\left( \sum_{n=0}^W X(\Phi_n) \right) ,
\]
where $X:=\delta f\circ f^{-1}$, $X(\cdot)$ is to differentiate in the direction of $X$, and $\Phi_n = \Phi\circ f^n$.
Linear response formulas are proved to give the correct derivative for various hyperbolic systems \cite{Ruelle_diff_maps,Dolgopyat2004,Jiang06,Llave86}, yet it fails for some cases \cite{Baladi2007,Wormell2019}.

Different linear response formulas may have very different numerical cost behaviors.
But no previous algorithms could afford to run on high-dimensional deterministic hyperbolic chaos.
To get a formula with better cost, decompose the linear response,
\[ \begin{split} \label{e:ruelle22}
  \delta \rho (\Phi)  = SC - UC ,
\end{split} \]
which we call the shadowing and the unstable contribution; this decomposition is straightforward in discrete-time.
Like a Leibniz rule, the shadowing contribution accounts for the change of the location of the attractor, while the unstable contribution accounts for the change of the measure.

The shadowing vector gives the shadowing contribution of the linear response for chaotic systems,
\begin{equation} \begin{split} \label{e:sc}
  SC = \rho(d\Phi \, S(\delta f\circ f^{-1})).
\end{split} \end{equation}
This utility is similar to conventional tangent solutions, which computes the linear response, or the parameter derivative, of stable dynamical systems.

\section{Adjoint shadowing for discrete-time}
\label{s:adjoint1}

This section first shows that the system of pulling-back covectors is also hyperbolic.
Then we prove the equivalent characterizations of the adjoint shadowing operator $\cS$, first on a hyperbolic path, then on an axiom A attractor with the physical measure.
Then we discuss the utility of each characterization.

\subsection{Adjoint hyperbolicity}
\hfill\vspace{0.1in}
\label{s:adjH}

This section shows that the iteration of the pullback operator $f^*$ is also hyperbolic.
Here $f^*$ is the adjoint of $f_*$, that is, for any $x\in\cM$, $w\in T_{x}\cM$, and $\eta \in T^*_{fx}\cM$,
$f^*$ is the operator such that
\[ \begin{split}
  \eta(f_{*} w) =  f^* \eta (w).
\end{split} \]
When $\cM=\R^M$, $\eta$ is a column vector and $f^* \eta = (Df)^T \eta$.
Moreover, define the norm on covectors
\[ \begin{split}
  |\eta| = \sup_w \eta(w)/ |w|.
\end{split} \]

Define the adjoint projection operators, $\cP^u$ and $\cP^s$,
such that for any $w\in T_{x}\cM$, $\eta \in T^*_{x}\cM$,
\[ \begin{split}
  \eta (P^u w) = \cP^u \eta(w),
  \quad \textnormal{} \quad 
  \eta (P^s w) = \cP^s \eta(w).
\end{split} \]
Define the image space of $\cP^u$ and $\cP^s$ as $V^{*u}$ and $V^{*s}$, then for any $v^u\in V^u$, $v^s\in V^s$, $\nu^u\in V^{*u}$, and $\nu^s\in V^{*s}$,
\[ \begin{split}
\nu^u v^s = \nu^s v^u = 0.
\end{split} \]
The next lemma shows that $V^{*u}$ and $V^{*s}$ are in fact unstable and stable subspaces for the adjoint system.
Hence, the adjoint system is also hyperbolic, and the adjoint covariant subspaces are dual to the tangent ones.

\begin{lemma} [adjoint hyperbolicity] \label{l:adj hyp}
There are $C>0$, $0<\lambda<1$, such that for any $x\in K$, $\eta\in T^*_x\cM$,
$n\ge0$, we have
\[ \begin{split}
|{f^*}^{-n} \cP^u \eta|\le  C \lambda^n |\cP^u \eta| , 
\quad \textnormal{} \quad 
|{f^*}^{n} \cP^s \eta| \le  C \lambda^n |\cP^s \eta|.
\end{split} \]
\end{lemma}

\begin{remark*}
(a)  
Notice that $f^*$ moves covectors backwards in time.
(b)
By this lemma, if we pullback $u$-many randomly-initialized covectors for many steps, 
they automatically occupy the unstable subspace,
since their unstable parts grow while stable parts decay.
\end{remark*}

\begin{proof}
Since $K$ is compact, there is constant $C>0$, such that for any $x_{-n}\in\cM$, any $w_{-n}\in T_{x_{-n}}\cM$, 
$|P^s w_{-n}| \le C|w_{-n}|$. Hence
\[ \begin{split}
  |{f^*}^{n} \cP^s \eta| 
  = \sup_{w_{-n}} \frac{|{f^*}^{n} \cP^s \eta(w_{-n})|} {|w_{-n}|}
  = \sup_{w_{-n}} \frac{ |\cP^s \eta( f_*^n P^s w_{-n})| }{ |w_{-n}|} \\
  \le C \lambda^n |\cP^s \eta| \sup_{w_{-n}} \frac{  |P^s w_{-n}|} {|w_{-n}|}
  \le C \lambda^n |\cP^s \eta| .
\end{split} \]
Here each $C$ is uniform on $K$, but its exact value may change from line to line.
The statement on $\cP^u$ can be proved similarly.
\end{proof}

\subsection{Pathwise adjoint shadowing lemma}
\hfill\vspace{0.1in}
\label{s:pathwiseasl}

In classical dynamical system theory, the shadowing lemma can be stated on a path on a hyperbolic set \cite{Wen2016}.
In particular, it does not assume axiom A nor the existence of physical measures.
In \cref{t:AS} we assume axiom A so that characterization (a) is more informative and we can prove (a) $\Rightarrow$ (b).
This subsections states a pathwise version of the lemma which does not assume axiom A, but then (a) is less informative.
For generality, we use the discrete topology on all steps of an orbit $\cO$, in particular, if $x_1=x_2$, they are still counted as two disjoint points.

\begin{lemma} [pathwise adjoint shadowing lemma] 
\label{l:pathasl}
On any orbit $\cO:=\{x_n\}_{n\in\Z}$ on a compact hyperbolic attractor, 
for any bounded covector sequence $\om\in \cX^*(\cO)$, or $\{\om_n\in T^*_{x_n} \cM\}_{n\in\Z}$,
the following characterizations of the shadowing covector sequence $\nu\in \cX^*(\cO)$ of $\om$ have relation (c) $\Leftrightarrow$ (b) $\Rightarrow$ (a).

  \begin{enumerate}
  \item 
  For any bounded vector sequence $X\in \cX(\cO)$,
  \[ \begin{split}
  \lim_{T\rightarrow\infty} 
  \frac 1{2T+1} 
  \sum _{n=-T}^T \om_n S(X)_n
  -
  \frac 1{2T+1} 
  \sum _{n=-T}^T \nu_n X_n
  =0.
  \end{split} \]
  \item 
  $\nu$ has the expansion formula given by a `split-propagate' scheme,
  \begin{equation*}
    \nu = \sum_{n\ge 0} f^{*n} \cP^s \om_n
    -  \sum_{n\le -1}  f^{*n} \cP^u \om_n\,.
  \end{equation*}
  \item 
  $\nu$ is the unique bounded solution of the inhomogeneous adjoint equation,
  \end{enumerate}
  \[ \begin{split}
  \nu_n = f^* \nu_{n+1} + \om_n.
  \end{split} \]
\end{lemma}

We define the pathwise adjoint shadowing operator, a linear operator $\cS:\cX^*(\cO) \rightarrow \cX^*(\cO)$, by $\cS(\om)=\nu$ in (b) or (c).
To get (b) from (a), we need more assumptions on the recurrence of all points on $\cO$.
For example, if $\cO$ is periodic, or if $\cO$ is a regular path for the physical measure and $\nu$ is continuous, then we can prove (a) $\Rightarrow$ (b): this is related to the statement of \cref{t:AS}.

The expansion of the shadowing covector in (b) is the `split-propagate' scheme applied on covectors.
That is, first decompose $\om$ according to the adjoint hyperbolicity, propagate stable components to the past, unstable components to the future; then $\nu$ is the summation of the covectors propagated to the step of interest.
This is similar to the conventional adjoint solution, $\nu'$, which is the sum of $\om$ from only the future.
Both procedures are illustrated in figure~\ref{f:nu}.

\begin{figure}[ht] \centering
  \includegraphics[scale=1]{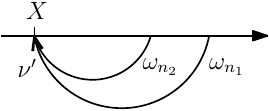}\hfill
  \includegraphics[scale=1]{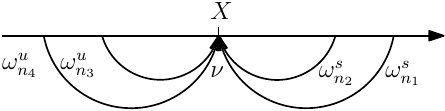}
  \caption{Illustrations for inhomogeneous adjoint solution with zero initial condition (left), and the shadowing covector (right).}
  \label{f:nu}
\end{figure}

The characterization in (c) is very similar to the conventional adjoint method.
The governing equation is the same, and the only difference is the extra boundedness requirement.
Boundedness is free in stable systems where all perturbations decay; in chaos boundedness requires more cost to achieve.
The similarity indicates that the adjoint shadowing covector is also related to the parameter derivatives of the dynamical system, which shall be more clear once we introduce the physical measure.

\begin{proof}
(b) $\Rightarrow$ (c).
  $\nu$ is bounded on the entire attractor $K$ and on every orbit,
  due to the exponential decay in the adjoint hyperbolicity.
  For the governing equation,
\[ \begin{split}
  & \nu  - f^* \nu_1\\
  =& \sum_{n\ge 0} f^{*n} \cP^s \om_n
  - \sum_{n\le -1}  f^{*n} \cP^u \om_n
  - \sum_{n\ge 0} f^{*n+1} \cP^s \om_{n+1}
  + \sum_{n\le -1}  f^{*n+1} \cP^u \om_{n+1} \\
  =& \sum_{n\ge 0} f^{*n} \cP^s \om_n
  - \sum_{n\le -1}  f^{*n} \cP^u \om_n
  - \sum_{k\ge 1} f^{*k} \cP^s \om_{k}
  + \sum_{k\le 0}  f^{*k} \cP^u \om_{k} \\
  =& \cP^s \om +\cP^u \om = \om.
\end{split} \]

(c) $\Rightarrow$ (b).
For uniqueness, pick any orbit and any two bounded inhomogeneous adjoint solutions, and let $\nu''$ be their difference.
Then $\nu''$ is a bounded homogeneous adjoint solution.
If $\nu''$ is not zero, then it grows exponentially fast either in the positive or negative time direction, so $\nu''$ can not be bounded.

(b) $\Rightarrow$ (a).
We want to show $\lim P_1-P_2=0$, where
\[ \begin{split}
  P_1 := \frac 1{2T+1} \sum_{n=-T}^T \om_n S(X)_n,
  \quad \textnormal{} \quad 
  P_2 := \frac 1{2T+1} \sum_{n=-T}^T \nu_n X_n.
\end{split} \]
Define a pivot quantity $P_3$, which is a sum involving only $\om_n$ and $X_n$ for $|n|\le T$,
\[ \begin{split}
  P_3 := \frac 1{2T+1} \sum_{n=-T}^T 
  \om_n \left( \sum_{0\le k \le n+T } f^{k}_{*} P^s X_{n-k} 
  - \sum_{n-T\le k\le-1} f^{k}_{*} P^u X_{n-k} \right),
\end{split} \]
Substitute the expression of $S(X)$ into $P_1$
\[ \begin{split}
  P_1
  = \frac 1{2T+1} \sum_{n=-T}^T 
  \om_n \left( \sum_{0\le k } f^{k}_{*} P^s X_{n-k} 
  - \sum_{ k\le-1} f^{k}_{*} P^u X_{n-k} \right),
\end{split} \]
Hence
\[ \begin{split}
  P_1 - P_3
  = \frac 1{2T+1} \sum_{n=-T}^T 
  \om_n \left( \sum_{k \ge n+T+1 } f^{k}_{*} P^s X_{n-k} 
  - \sum_{k\le n-T-1} f^{k}_{*} P^u X_{n-k} \right),\\
  \Rightarrow 
  |P_1 - P_3| 
  \le \frac 1{2T+1} \sum_{n=-T}^T 
  \left( \sum_{k \ge n+T+1 }C\lambda^{k}
  - \sum_{k\le n-T-1} C\lambda^{|k|} \right) \\
  \le C \frac 1{2T+1} \sum_{n=-T}^T \lambda^{n+T+1}+\lambda^{T+1-n}
  \le C \frac 1{2T+1} \frac 1{1-\lambda}
  \rightarrow 0,
  \quad \textnormal{as} \quad 
  T\rightarrow \infty.
\end{split} \]
The difference between $P_2$ and $P_3$ is similar.
\end{proof}

\subsection{Adjoint shadowing lemma}
\hfill\vspace{0.1in}
\label{s:december}

Once we have the physical measure $\rho$, we can define the product by integration of $\rho$, then we can prove (a) $\Rightarrow$ (b), so all three characterizations are equivalent.
Moreover, the new product gives an expression of a part of linear response.

\goldbach*

The unique existence of the adjoint operator in (a) follows from its equivalence to (b) and (c).
Since the operator in (a) is not in Hilbert spaces, we can not get its unique existence from the textbook knowledge on adjoint operators.
We may as well define $S$ and $\cS$ as operators on $L^2(\rho)$, but many statistical application still require Holder.

\begin{proof}
The proof of (b) $\Rightarrow$ (c) is the same as the pathwise \cref{l:pathasl}.

(c) $\Rightarrow$ (b).
For any point $x\in K$, if another vector field satisfies the adjoint equation but does not equal $\cS(\om)$, then it must go to infinity either in the past or the future of $x$. 
But this contradict our overall assumption that $\cS(\om)$ is continuous on $K$ so it must be bounded.

(b) $\Rightarrow$ (a).
First prove the duality.
We can prove via the pathwise version lemma, or somewhat similarly but more easily, substitute the expansion formula of $S$ in \cref{e:expand_v} into the shadowing contribution in~\cref{e:sc}, then move all operations to $\om$, apply the invariance of SRB measures, 
\[ \begin{split}
  \rho(\om S(X)) 
  = \rho\left( \om \sum_{n\ge0} f^{n}_{*} P^s X_{-n} 
  - \om \sum_{n\le-1} f^{n}_{*} P^u X_{-n} \right)\\
  = \rho\left( \sum_{n\ge0} (f^{*n} P^s \om) X_{-n} 
  - \sum_{n\le-1} (f^{*n} P^u \om) X_{-n} \right) \\
  = \sum_{n\ge0}\rho\left(  (f^{*n} P^s \om_n) X \right)
  - \sum_{n\le-1}\rho\left( (f^{*n} P^u \om_n) X\right)\\
  = \rho\left( (\sum_{n\ge0} f^{*n} P^s \om_n 
  - \sum_{n\le-1} f^{*n} P^u \om_n) X\right) 
  = \rho(\cS(\om) X) .
\end{split} \]

(a) $\Rightarrow$ (b).
If there is another $\cS'$ satisfying the duality in (a), then 
\[ \begin{split}
  \rho((\cS(\om)-\cS'(\om)) X)=0
\end{split} \]
for any $X\in \cX^\alpha$. 
Hence $\cS(\om)=\cS'(\om)$ $\rho$-almost everywhere.

Appendix~\ref{a:holder} proves that $\cS$ preserves Holder continuity.
\end{proof}

\subsection{Applications of adjoint shadowing lemma}
\hfill\vspace{0.1in}
\label{s:appdiscrete}

The three characterizations are useful for different purposes. 
Roughly speaking, (a) and (b) tell us what $\cS$ could be used for once we can compute it, whereas (c) tells us how to compute $\cS$.

Characterization (c) can be efficiently and conveniently recovered by nonintrusive shadowing algorithms \cite{Ni_nilsas}.
More specifically, to compute shadowing covector $\nu$ of $\om$, we solve the nonintrusive adjoint shadowing problem on an orbit of step  $0, 1, \ldots, N$,
\[ \begin{split} 
  \nu =  \nu ' + \ueps a \,,
  \quad  \mbox{s.t. }
  \ip{\nu_0, \ueps_0}=0.
\end{split}\]
Here $\nu'$ is the inhomogeneous adjoint solution;
$\ueps$ is a matrix with $u$ columns of homogeneous adjoint solutions with $\ueps_N$ randomly generated,
\[ \begin{split}
  \nu' = f^* \nu'_1 + \omega,
  \quad 
  \nu'_{N}=0;
  \quad \textnormal{} \quad 
  \ueps = f^* \ueps_1.
\end{split} \]
Intuitively, the main cause of the growth of $\nu'$ is that $\nu'-\nu$ contains unstable components.
During the pullback, the span of $\ueps$ becomes a good approximation of the unstable subspace, so we can tame the growth of $\nu'$ by subtracting a linear combination of columns of $\ueps$.
There are various ways to find the linear combination, for example, if $\nu\perp V^{*u}$ at step $0$, then $\nu$ can not have a significant unstable part most of the time.
The proof of the convergence of the algorithm is in \cite{Ruesha}.

The nonintrusive shadowing algorithm is easy to implement and efficient.
It is very similar to conventional backpropagation algorithms, where only one inhomogeneous adjoint solution, $\nu'$, is solved.
Nonintrusive shadowing runs the conventional backpropagation solver $u$ more times to get $\ueps$.
Solving extra adjoint equations are not as expensive as solving the first one, since the most expensive cost for solving these linear ODEs is loading the Jacobian matrix $f^*$.

Characterization (a) in the adjoint shadowing lemma tells us that we can use $\cS$ to compute $SC$ `adjointly' in the sense of utility, which means that major computations are away from $X$, so a new parameter (hence a new $X$) does not require much re-computation.
In numerical practice, `adjoint' also sometimes means the computation involves only the pull back of covectors: that is also true for this paper, but not for the equivariant divergence formula, which is the unstable part of the linear response.

When $u\ll M$, it is likely that the unstable projection tends to be small, so the unstable part of the perturbation tends to be small; moreover, if the system has fast decay of correlation, then $UC$ tends to be small, and $SC$ can be a good approximation of the entire linear response \cite{Ruesha}, which can be expressed by $\cS$.
However, we must notice that this estimation is based on statistical assumptions which are not universal, and we can construct many examples which violate these assumptions.
This estimation is mainly an explanation and expectation of what we observe in practice, for example in \cite{Ni_CLV_cylinder}. 
The morale is that, when $u\ll M$, we should notice that we can perhaps already get something useful with a simple extension of the numerical tools we already have.
When the unstable contribution is large, or when better accuracy is desired, for example near design optimal, we need to further compute the unstable contribution \cite{RepolhoCagliari2021,Lasagna2019,BloniganPhdThesis}.

Because the `split-propagate' scheme in (b) also appears in the expansion formula of $UC$, the adjoint shadowing lemma is also important for computing the unstable contribution. 
Our work on shadowing will not be wasted if we later want a precise algorithm for the entire linear response.
More specifically,
\[ \begin{split}
  UC = \lim_{W\rightarrow\infty} - \rho(\phi \frac{\delta^u \tL^u\sigma} \sigma),
  \quad\textnormal{where}\quad 
  \phi:=\sum_{m=-W}^W (\Phi\circ f^m -\rho(\Phi)).
\end{split} \]
Here $\sigma$ is the conditional measure on unstable manifolds, and $\delta^u \tL^u\sigma$ is the unstable perturbation of the transfer operator on $\sigma$, which has the equivariant divergence formula,
\[ \begin{split}
 - \frac{\delta^u \tL^u\sigma} \sigma =  \div^u_\sigma X^u = \cS(\div^v f_*) X + \div^v X,
 \quad \textnormal{where} \quad 
 X = \delta f\circ f^{-1}.
\end{split} \]
Roughly speaking, $\div^v$, the so-called `equivariant divergence', is the contraction by the unit unstable cube and its co-cube; the covector $\div^vf_*$ is the divergence of the Jacobian matrix $f_*$.
This formula is also `adjoint' in the utility sense, since major computations are away from $X$ and $\nabla X$.
This formula can be computed by $2u$ recursive relations on an orbit: this is in some sense the ergodic theorem or Monte-Carlo formula for the linear response.
The detailed theory for the equivariant divergence formula for discrete-time is in \cite{TrsfOprt}; the discrete-time algorithm and numerical examples are in \cite{far}.

\section{Preparations for continuous-time}
\label{s:prep2}

\subsection{Hyperbolicity and tangent shadowing}
\hfill\vspace{0.1in}
\label{s:hypershaF}

For continuous-time, let $f$ be the flow of $F+\gamma X$, where both $F$ and $X$ are smooth vector fields, $F$ is the base flow, $X$ is the perturbation, and the parameter $\gamma$ has base value zero.
Assume that 
\[ \begin{split}
  F\neq 0 \quad \textnormal{on } K.
\end{split} \]
Since the attractor $K$ is compact, $|F|$ is uniformly bounded away from zero and infinity.
Now $T_KM$ has a continuous $f_*$-invariant splitting into stable, unstable, and center subspaces, $T_KM = V^s \bigoplus V^u \bigoplus V^c$,
Where $V^c$ is the one-dimensional subspace spanned by $F$.
Should we have $F=0$ at some points, then the hyperbolicity of the entire system may be ruined and our works can be affected to various degrees: this requires some new techniques, which we shall discuss in \cref{s:future}.

The notations in continuous-time is more convoluted.
For an orbit $x_t=f^t(x_0)$, a homogeneous tangent equation of $e_t\in T_{x_t}\cM$ can be written in three ways:
\[ \begin{split}
  e_t  =f_*^t e_0 
  \quad\Leftrightarrow\quad \cL_F e = 0
  \quad\Leftrightarrow\quad \nabla_F e = \nabla_e F.
\end{split} \]
Here $\cL$ is the Lie-derivative and $\nabla$ is the Riemannian derivative.
The last expression is an ODE since $\nabla_F$ is typically denoted by $\partial/\partial t$ in $\R^M$.

To define the tangent shadowing operator, for any orbit $\cO$ with topology of $\R$, (that is, a self-cross is counted as two points), define $A(\cO)$ as the quotient space
\begin{equation} \begin{split} \label{e:shan}
  A(\cO):= \{ (v,\eta): v\in \cX^1(\cO), \,
  \eta \in C(\cO)
  \} \, /\, \sim.
\end{split} \end{equation}
Here $\cX^1(\cO)$ is the space of bounded vector fields continuously differentiable along $\cO$. 
The equivalent relation $\sim$ is defined as
\[ \begin{split}
(v_1,\eta_1)\sim (v_2,\eta_2)
\quad \textnormal{iff} \quad 
  \cL_F v_1 +\eta_1 F = \cL_F v_2 +\eta_2 F.
\end{split} \]
We can further define $A$ over $K$ as the quotient space
\[ \begin{split}
  A(K):=\{ (v,\eta): v\in \cX^\alpha (K), \,
  \nabla_F v\in  \cX^\alpha (K), \,
  \eta\in C^\alpha(K) \} \, /\, \sim.
\end{split} \]
Note that $A(K)$ has higher regularity than $A(\cO)$, since we can talk about better regularities on $K$.

We define the shadowing operator 
$ S_\cO: \cX(\cO) \rightarrow A(\cO)$
and $ S_K: \cX^\alpha(K) \rightarrow A_K$ as
\[ \begin{split}
S: X \mapsto [v, \eta],
\quad \textnormal{where} \quad 
\cL_F v + \eta F
= X .
\end{split} \]
We call $(v,\eta)$ a `shadowing pair' of $X$.
Here $[v,\eta]$ is the equivalent class of $(v,\eta)$ according to $\sim$.
The map is well-defined, due to the definition of $\sim$.
We sometimes omit the subscript of $S$ and bracket following $A$ and $\cX$, when there is no ambiguity, or when the argument applies to both cases.

Intuitively, $v$ points to the shadowing orbit and $\eta$ is the time-rescaling: this is intuitively explained in \cite[appendix C]{Ni_NILSS_JCP}, whose $\eta$ is the negative of this paper.
In flows, the shadowing orbit, which lies close to the original orbit, is unique as a set of points.
However, the points on shadowing orbits may not move at the same speed, and we need to reparameterize time, so that the points on shadowing orbits are close for all time.
This time reparameterization is bounded but not unique \cite[theorem 18.1.7]{Katok_thick_book}.
Hence the pair $(v,\eta)$ is not unique.
For example, adding $F$ to $v$ gives another shadowing pair.

We write out the expression of a specific shadowing pair that fits the definition,
\begin{equation} \begin{split} \label{e:haohao}
  \eta F = P^c X,
  \quad \textnormal{} \quad 
  v 
  = \int_{t\ge0} f^{t}_{*} P^s X_{-t} 
  - \int_{t\le0} f^{t}_{*} P^u X_{-t} .
\end{split} \end{equation}
It follows that $\eta = \eps^c X$ so it is Holder on $K$, where $\eps^c$ is the unit center covector;
$v$ is Holder due to \cref{a:holder}.
Then $\cL_F v = X - \eta F$ is Holder.
Since $\nabla_v F$ is Holder, $\nabla_F v = \cL_F v +\nabla_v F $ is also Holder.
Hence, $[v,\eta] \in A(K)$.

The nonintrusive shadowing algorithm has been applied to a $4\times10^6$ dimensional system in computational Navier-Stokes equations \cite{Ni_NILSS_JCP,Ni_fdNILSS,Ni_CLV_cylinder}.
It does not solve for the shadowing pair in~\cref{e:haohao}, rather, it solves for the orthogonal projection of $v$, and the integration of the corresponding $\eta$.
More specifically, it solves
\[ \begin{split} 
  v = v' + \sum_{i=1}^u e^i a_i \,,
  \quad \mbox{s.t. }
  \ip{v_T^\perp, e^{i\perp}_T}=0
  \quad \textnormal{for all} \quad 
  1\le i\le u.
\end{split}\]
Here $v'$ is a particular inhomogeneous adjoint solution, and $(\cdot)^\perp$ is the orthogonal projection onto the subspace perpendicular to $F$.
It was not very clear that its result is the same with other shadowing pairs and gives a significant part of (but not the whole) the linear response.
Section~\ref{s:sudcomp} shall provide positive answers.

\subsection{Physical measure and linear response}
\hfill\vspace{0.1in}
\label{s:S}

Under similar assumptions as discrete-time, the attractor admits a physical SRB measure $\rho$, and similar linear response theorem has been proved \cite{Ruelle_diff_flow,Dolgopyat2004}.
In particular, the linear response in continuous-time is the same,
\begin{equation} \begin{split} \label{e:meiguo}
  \delta \rho (\Phi)
  = \lim_{T\rightarrow\infty} \rho\left( \int_{0}^T X(\Phi_t) dt \right) ,
\end{split} \end{equation}
assuming the decay of correlation, 
\begin{equation}
\label{e:deep}
  \int_0^\infty \rho  (\Phi_t C) dt < \infty.
\end{equation}
Here $\Phi_t = \Phi\circ f^t$, $C=\div^{cu}_\sigma X^{cu}$ is the submanifold divergence on the center-unstable manifold under conditional SRB measure.

The shadowing/unstable decomposition of the linear response in continuous-time is a little tricky since it is not apparent which part should have the center direction.
This issue is resolved in section~\ref{s:sudcomp}.
We shall also assume another decay of correlation
\begin{equation} \begin{split} \label{e:published}
    \lim_{T\rightarrow\infty} \rho\left( \eta \Phi_T \right) = \rho(\eta)\rho(\Phi)
    \quad \textnormal{for any Holder continuous } \eta .
\end{split} \end{equation}
Both decorrelation can be proved under various assumptions \cite{Liverani04,Dolgopyat1998}, but we shall just assume them in this paper.

\section{Shadowing/unstable decomposition for continuous-time}
\hfill\vspace{0.1in}
\label{s:sudcomp}

We define the shadowing/unstable decomposition for the linear response of flows.
Due to the presence of the center direction and the non-uniqueness in the shadowing pair, this decomposition is not straightforward and was not previously defined.
When defining the shadowing contribution, we want to give positive answers to the two historically lingering questions: 
\begin{enumerate}
  \item Do different shadowing pairs in the same equivalent class in $A$, defined in \cref{e:shan}, give the same shadowing contribution?
  \item Is the shadowing contribution a significant part of the linear response?
\end{enumerate}

\subsection{Well-definedness of a product between \texorpdfstring{$A$}{A}
and \texorpdfstring{$\cA$}{curly A}}
\hfill\vspace{0.1in}
\label{s:sleep}

On a hyperbolic set, for an orbit $\cO$ on with real-line topology, define the space of pairs of a covector field and a scalar function differentiable along $F$,
\begin{equation} \begin{split} \label{e:yy}
  \cA (\cO):=\{(\om, \psi) \,|\, \om\in \cX^{*}(\cO), \psi\in C^1(\cO), F(\psi) = \om(F)\}.
\end{split} \end{equation}
Here $\cX^*$ is the space of continuous bounded vector fields on $\cO$.
We can also define $\cA$ over $K$ with higher regularity
\[ \begin{split}
  \cA (K):=\{(\om, \psi) \,|\, 
  \om\in \cX^{*\alpha}(K), 
  \psi\in C^\alpha(K), 
  F(\psi) \in C^\alpha(K), 
  F(\psi) = \om(F)\}.
\end{split} \]
The main case in this paper satisfies $\om=d\psi$ and hence the constraint in $\cA$, but there are cases where the differential holds only in the flow direction, for example in $UC$: this will be clear in a forthcoming paper.

The products between $[v,\eta]\in A$ defined in \cref{s:hypershaF}, and $(\om,\psi)\in \cA$ are
\[ \begin{split}
  \llangle [v,\eta];\om,\psi\rrangle _{\cO,T}
  :=\frac 1{2T} \int_{-T}^T \om v - \eta \psi dt ,
  \quad \textnormal{} \quad 
  \llangle [v,\eta];\om,\psi\rrangle _\cO
  := \lim_{T\rightarrow\infty}
  \llangle [v,\eta];\om,\psi\rrangle _{\cO,T},\\
\quad \textnormal{} \quad 
  \llangle [v,\eta];\om,\psi\rrangle _K 
  := \rho (\om v) - \rho(\eta \psi).
\end{split} \]

We first prove that the product is well-defined.
Hence, very roughly speaking, $A$ and $\cA$ are dual under the product: their `dimensions' match, since $A$ is $\cX \times C$ modulo an equivalent relation, and $\cA$ is $\cX \times C$ with a constraint.
However, for $A$ and $\cA$ to be true dual, we have to reduce the regularity from Holder to dual regularities; we shall not do that in this paper.

\begin{lemma} \label{l:xue}
For $(\om,\psi)\in\cA$, all $(v,\eta)$ in a equivalent class in $A$ give the same product $\llangle v,\eta; \om,\psi\rrangle$ on $\cO$ and on $K$.
Hence $\llangle S(X); \om,\psi\rrangle$ is also well-defined.
\end{lemma}

\begin{proof}
  Let $(v_1,\eta_1)$ and $(v_2,\eta_2)$ be two shadowing pairs of $X$.
  Denote $(\cdot)'':=(\cdot)_1-(\cdot)_2$.
  By definition, $v''$ and $\eta''$ are both bounded, and
  \[ \begin{split}
    \cL_F v'' = 
    \nabla_F  v'' - \nabla_{ v''} F
    = - \eta'' F.
  \end{split} \]

  Decompose $v'' = v^{su}+v^c$, where $v^{su}\in V^{su} := V^s\bigoplus V^u$, $v^c\in V^c$.
  Due to the invariance of the subspaces, $\cL_F v^{su} \in V^{su}$, $\cL_F v^{c} \in V^{c}$, so $\cL_F v^{su}=0$.
  However, due to the exponential growth, $v^{su}$ must be zero, otherwise it grows unbounded.
  Hence, $v'' \in V^c$, so we can write it as 
  \[ \begin{split}
   v'' = -hF ,
   \quad \textnormal{where} \quad 
   \partial h/\partial t = \eta''.
  \end{split} \]
  By assumptions, $v''$ is bounded, $|F|>C>0$, so $h$ is bounded.
  This basically means that $\eta''$ has a zero average, otherwise $h$ would be unbounded.

  To prove the lemma on an orbit $\cO$, note the difference in the product is
  \[ \begin{split}
   \llangle v'', \eta''; \om, \psi \rrangle 
    = \lim_{T\rightarrow\infty} \frac 1{2T} \left( \int_{-T}^T(\om \, v'') dt 
    - \int_{-T}^T \psi \eta'' dt 
     \right).
  \end{split} \]
  Since $\om v'' = -h\om F = -hF(\psi) = -h \partial\psi/\partial t$,
  \[ \begin{split}
   \llangle v'', \eta''; \om, \psi \rrangle 
    = \lim_{T\rightarrow\infty} \frac 1{2T} \left( - \int_{-T}^T h (\partial \psi/\partial t) dt 
    - \int_{-T}^T \psi (\partial h/\partial t) dt  \right)\\
    = \lim_{T\rightarrow\infty} \frac 1{2T} \left( h(-T)\psi(-T)- h(T)\psi(T) \right)
    = 0,
  \end{split} \]
  since $h$ is bounded.

  To prove the lemma on $K$, take an orbit such that its long-time-average of $\om v$ and $\psi\eta$ equal their integrations to $\rho$, and change the lower bound of integration from $-T$ to time zero.
\end{proof}

\subsection{Unstable contribution is in unstable subspace}
\hfill\vspace{0.1in}
\label{s:waiyi}

\begin{definition}
\label{d:scflow}
On a compact mixing axiom A attractor with physical measure $\rho$ and decay of correlations in \cref{e:deep,e:published}, 
define the shadowing contribution in the linear response of flows as 
\[ \begin{split}
  SC :=\llangle S(X); d\Phi, \Phi -\rho(\Phi) \rrangle_K,
\end{split} \]
where $X=\delta f$.
\Cref{l:xue} shows that this is well-defined.
\end{definition}

We want the leftover part to be in the unstable direction.
If so, we can recycle the dimension argument in~\cite{Ruesha}, which is also reviewed in \cref{s:appdiscrete}, to say that the shadowing contribution could be a good approximation of the entire linear response when the unstable ratio is low.
We do not want $UC$ to involve the flow direction, which could be more easily captured than unstable and stable directions.
For example, a perturbation in the time equals a perturbation in the flow direction.
If $UC$ precludes the flow direction, it is more likely to be small, and $SC$ alone is more likely to be a good approximation of the linear response: sometimes this is indeed the case in practice, such as in \cite{Ni_CLV_cylinder}.

\begin{theorem} [shadowing/unstable decomposition of flows] \label{t:region}
\Cref{d:scflow} equals
  \[ \begin{split}
    SC
    = 
    \int_{0}^\infty \rho\left(X^{sc}(\Phi_t) \right) dt
    - \int_{-\infty}^0 \rho\left(X^u(\Phi_t)  \right) dt.
  \end{split} \]
  Hence, the unstable contribution, the leftover part of the linear response in \cref{e:meiguo}, is
  \[ \begin{split}
    - UC := 
    \delta \rho(\Phi) - SC 
    = \int_{-\infty}^\infty \rho\left(  X^u(\Phi_t) dt \right).
  \end{split} \]
\end{theorem}

\begin{proof}
  First let $v$ and $\eta$ be the specific shadowing pair in \cref{e:haohao}.
  By the invariance of $\rho$, the stable and unstable part of the expression of $SC$ in the lemma equals
  \[ \begin{split}
    \int_{0}^\infty \rho\left(X^{s}(\Phi_t) \right) dt
    - \int_{-\infty}^0 \rho\left(X^u(\Phi_t)  \right) dt \\
    = \rho\left( \left(\int_{0}^\infty f_*^t X^{s}_{-t}dt - \int_{-\infty}^0  f_*^t X^u_{-t} dt \right) (\Phi)\right) 
    = \rho(v(\Phi)).
  \end{split} \]
  For the center direction, note $F=\partial/\partial t$ and the decorrelation in \cref{e:published}, 
  \begin{equation} \begin{split} \label{e:duang}
    \int_{0}^\infty \rho\left(X^{c}(\Phi_t) \right) dt
    = \lim_{T\rightarrow\infty} \rho \left( \eta \int_{0}^T F(\Phi_t) dt \right)  
    = \lim_{T\rightarrow\infty} \rho \left( \eta \int_{0}^T \pp{\Phi_t}t dt \right)  \\
    = \lim_{T\rightarrow\infty} \rho \left( \eta (\Phi_T-\Phi) \right) 
    = \rho (\eta)\rho(\Phi) - \rho (\eta\Phi)
    = - \rho (\eta(\Phi - \rho(\Phi))).
  \end{split} \end{equation}
  Hence the expression in the lemma equals the definition of $SC$ in \cref{d:scflow} with the specific $v$ and $\eta$.
  By lemma~\ref{l:xue}, we further know that the choice of the shadowing pair does not matter.
\end{proof}

By \cref{e:duang}, it turns out the center direction has zero average
\[ \begin{split}
  \int_{-\infty}^\infty \rho\left(  X^{c}(\Phi_t) dt \right) = 0
\end{split} \]
Hence, $UC$ has another expression involving the center direction, which is convenient for deriving the equivariant divergence formula of continuous-time: this will be clear in a forthcoming paper.
\begin{proposition} \label{l:ppp}
\[ \begin{split}
  - UC := 
  \delta \rho(\Phi) - SC 
  = \int_{-\infty}^\infty \rho\left(  X^u(\Phi_t) dt \right) 
  = \int_{-\infty}^\infty \rho\left(  X^{cu}(\Phi_t) dt \right) .
\end{split} \]
\end{proposition}

\section{Adjoint shadowing for continuous-time}
\label{s:adj2}

This section develops the adjoint shadowing theory for continuous-time systems.
The main issue is the extra center subspace given by the flow direction.
There is no exponential decay in the center subspace, but we can still obtain similar results.

\subsection{Adjoint hyperbolicity}
\hfill\vspace{0.1in}
\label{s:adjhyperF}

With some extra notations, we have the same results as the discrete-time.
Define the Riemannian derivative along a covector by:
\begin{equation} \begin{split} \label{e:deng}
  \nabla_\om X :=  \nabla_X\om - \cL_X\om ,
\end{split} \end{equation}
where $\cL$ is the Lie derivative.
By this definition,
\[ \begin{split}
  (\nabla_\om X) Y = (\nabla_X\om)Y - (\cL_X\om)Y
  = \nabla_X(\om Y) -(\nabla_XY)\om - \cL_X(\om Y) + (\cL_XY)\om .
\end{split} \]
Since $\om Y$ is a scalar function, its Lie derivative equals Riemannian derivative, hence,
\begin{equation} \begin{split} \label{e:hey}
  (\nabla_\om X) Y 
  =  (\cL_XY- \nabla_XY)\om
  = - \om \nabla_YX.
\end{split} \end{equation}
For an orbit $x_t=f^t(x_0)$, a homogeneous adjoint equation of $\eps_t\in T^*_{x_t}\cM$ can be written in three ways:
\[ \begin{split}
  f^{*t} \eps_t  = \eps_0 
  \quad\Leftrightarrow\quad \cL_F \eps = 0
  \quad\Leftrightarrow\quad \nabla_F \eps = \nabla_\eps F.
\end{split} \]
Here the pullback operator $f^{*t}$ is the adjoint of $f^t_*$.
The last expression is an ODE.

Define the adjoint projection operator, $\cP^u$, $\cP^s$, and $\cP^c$, similarly as discrete-time.
As lemma~\ref{l:adj hyp}, we can also show that their image spaces hyperbolically split the cotangent space.
In particular, since $F$ is non-zero on a compact set $K$,
\[ \begin{split}
  |{f^*}^{t} \cP^c \eta|\le  C |\cP^c \eta| ,
  \quad \textnormal{ for any }
  t\in \R.
\end{split} \]

\subsection{Pathwise adjoint shadowing lemma}
\hfill\vspace{0.1in}
\label{s:pathaslF}

\begin{lemma} [pathwise adjoint shadowing lemma] 
\label{l:pathaslF}
On an orbit $\cO:=\{x_t\}_{t\in\R}$ on a compact hyperbolic attractor, 
for any $(\om,\psi)\in \cA(\cO)$ defined in \cref{e:yy}, the following characterizations of the shadowing covector $\nu\in \cX^*(\cO)$ of $\om$ have relation (c) $\Leftrightarrow$ (b) $\Rightarrow$ (a).
\begin{enumerate}
\item 
  For any $X\in \cX (\cO)$, let $S:\cX(\cO) \rightarrow A(\cO)$ be the (tangent) shadowing operator, then
  \[ \begin{split}
  \lim_{T\rightarrow\infty} 
  \llangle \om, \psi ; S(X) \rrangle _{\cO,T}
  -
  \frac 1{2T} \int_{-T}^T X \nu dt
  =0.
  \end{split} \]
\item 
  $\nu$ has the expansion formula given by a `split-propagate' scheme,
  \begin{equation*}
    \nu = \int_{t\ge 0} f^{*t} \om^s_t dt
    -  \int_{t\le 0} f^{*t} \om^u_t dt
    - \psi \eps^c,
  \end{equation*}
\item 
  $\nu$ is the unique bounded solution of the inhomogeneous adjoint equation,
  \[ \begin{split}
  \nabla_F \nu - \nabla_\nu F 
  = \cL_F\nu
  = - \omega 
  \;\textnormal{ on } \cO,
  \quad \textnormal{} \quad 
  \nu_\tau(F_\tau) = -\psi_\tau
  \;\textnormal{ for all or any } \tau.
  \end{split} \]
\end{enumerate}
\end{lemma}

We define pathwise adjoint shadowing operator, a linear operator $\cS:\cA(\cO) \rightarrow \cX^{*} (\cO ) $, by $\cS(\om,\psi)=\nu$ in (b) or (c).
By (b) and (c), if $\nu$ satisfies the ODE for all time, and $\nu F = -\psi$ at one time, then $\nu F = -\psi$ for all time.

The expression in (b) is similar to the discrete-time case but has the extra center direction.
It is not apriorily clear that the extra center direction assembles well with the other directions, as now we have in (c).

In $\R^M$ the ODE in (c) is the conventional adjoint equation, $-\dd \nu t = DF^T\nu +\om$.
To see the similarity of (c) with discrete-time,  note that $f^*\nu_1 - \nu$ is the time-1 approximation of $\cL_F\nu $.

\begin{proof}
(b) $\Rightarrow$ (c).
  Boundedness is due to the adjoint hyperbolicity.
  For the governing equation, first write the expression of $\nu$ at $\tau$,
\[ \begin{split}
  \nu_\tau
  = \int_{t\ge \tau} f^{*t-\tau} \om^s_t dt
  - \int_{t\le \tau} f^{*t-\tau} \om^u_t dt
  - \psi_\tau \eps^c_\tau,
  \end{split} \]
  Note that for fixed $t$, $f^{*t-\tau}\om^s_t$ is a homogeneous adjoint solution, so
  \[ \begin{split}
    \nabla_\pp{} \tau f^{*t-\tau}\om^s_t
    = \nabla_{f^{*t-\tau}\om^s_t} F_\tau
  \end{split} \]
  We may also use Lie derivative to write this proof, but the notations with $\nabla$ is more consistent with the backpropagation, or the pathwise perturbation equation, in $\R^M$.
  Further note that the lower bound of the integration interval is $\tau$, so 
  \[ \begin{split}
    \nabla_\pp{} \tau \int_{t\ge \tau} f^{*t-\tau} \om^s_t dt
    =  \nabla_{\int_{t\ge \tau} f^{*t-\tau} \om^s_t dt} F_\tau -\om^s_t .
  \end{split} \]
  Similarly,
  \[ \begin{split}
    \nabla_\pp{} \tau \int_{t\le \tau} f^{*t-\tau} \om^u_t dt
    = \nabla_{\int_{t\ge \tau} f^{*t-\tau} \om^u_t dt} F_\tau + \om^u_t .
  \end{split} \]

  For the center direction, $\eps^c$ is also homogeneous, 
  \[ \begin{split}
    \nabla_\pp{} \tau \psi \eps^c 
    = \psi \nabla_\pp{} \tau \eps^c + \eps^c F(\psi) 
    = \psi \nabla_ {\eps^c } F + \eps^c \om(F).
  \end{split} \]
  For the last term, note that an element in the one-dimensional center subspace $V^{*c}$ is uniquely determined by its product with $F$. Since $\om(F)\eps^c \in V^{*c}$ and $\om(F)\eps^c(F) = \om(F)=\om^c(F)$, it follows that $\om(F)\eps^c = \om^c$.
  Hence,
  \[ \begin{split}
    \nabla_\pp{} \tau \psi \eps^c 
    = \nabla_ {\psi\eps^c } F + \om^c.
  \end{split} \]

  To summarize, $\nu$ satisfies the ODE
  \[ \begin{split}
  \nabla_F \nu
    = \nabla_{\int_{t\ge \tau} f^{*t-\tau} \om^s_t dt} F_\tau -\om^s_t 
    - \nabla_{\int_{t\ge \tau} f^{*t-\tau} \om^u_t dt} F_\tau - \om^u_t 
    - \nabla_ {\psi\eps^c } F - \om^c 
    = \nabla_\nu F -\om.
  \end{split} \]
To check the extra constraint, just notice that $f^{*t} \om^s\in V^{*s}$, so the integrations in $\nu$ has zero product with $F$.
So $\nu(F) = \psi \eps^c(F) = \psi$ for all time: this is the stronger version of (c).

(c) $\Rightarrow$ (b).
For uniqueness, assume the weaker version of (c), that is, pick any orbit and any two bounded solutions of the ODE and $\nu_\tau(F_\tau) = \psi_\tau$ at an arbitrary time $\tau$,
Then their difference $\nu''$ is a bounded homogeneous adjoint solution of $\cL_F \nu''=0$ such that $\nu''_\tau(F_\tau)=0$.
If $\nu''$ is not zero in the stable or unstable direction, then it can not be bounded.
Hence, $\nu'' = C\eps^c$ for some constant $C$.
But $\nu''=0$ at $\tau$, so the difference is always zero.

(b) $\Rightarrow$ (a).
Use the specific shadowing pair in \cref{e:haohao}.
The stable and unstable part is the same as discrete time.
For the center part, notice
\[ \begin{split}
  \eta = \eta F \eps^c = (P^c X) \eps^c = \eps^c X.
\end{split} \]
Also note $\nu^c = -\psi \eps^c$, so $ - \eta \psi = \nu^c X$ for all time.
\end{proof}

\subsection{Adjoint shadowing lemma}
\hfill\vspace{0.1in}
\label{s:aslF}

\silverbach*

By (a), we can say that $\cS$ is the adjoint of $S$.
We still have unique existence of $\cS$, but we do not prove that from from the standard adjoint operator theory for two reasons.
First, the standard theory is a bit technical to invoke for $A$ and $\cA$ to be rigorously dual; 
neither does it provide much benefit, in particular, it does not immediately give the Holder continuity.
Since we have (b) and (c), we have a more straightforward way to get unique existence and better regularity.

\begin{proof}[Proof of \cref{t:ASF}]

(b) $\Rightarrow$ (c). Same as the pathwise case.

(c) $\Rightarrow$ (b).
Assume $\nu\in \cX^*(K)$ satisfies the weaker version of (c), that is, $\cL_F\nu = -\om$ on $K$ and $\nu F = -\psi$ for one $x\in K$.
First, we can check that $\cS(\om,\psi)$ indeed satisfies all the conditions in (c), so the set of potential $\nu$'s is not empty.
Here $\cS$ temporarily refers to only the expression in (b).

Since $K$ is mixing axiom A, let $\cO$ be a transitive orbit on $K$, that is, all points on $K$ are limit points of $\cO$.
We can see that $\nu = \cS(\om, \psi) + C\eps^c$ for some constant $C$ on $\cO$.
Otherwise $\nu$ would grow unbounded, violating the overall assumption that $\nu\in \cX^*(K)$, or $\nu$ would not satisfy the center direction of the ODE (see the proof of the pathwise case).
Hence,  $\nu F = -\psi + C$ on $\cO$.

Note that $\cO$ does not necessarily contain $x$.
But $x$ is a limit point of $\cO$, so there are $\{x_n\}_{n\ge0} \subset \cO$ such that $x_n\rightarrow x$.
Note that by assumptions, $\nu, F$ and $\psi$ are continuous on $K$, so 
\[ \begin{split}
\nu F (x) \leftarrow \nu F (x_n) = -\psi (x_n) + C \rightarrow -\psi(x) + C.
\end{split} \]
By assumption, $C=0$.
So $\nu$ could only be $\cS(\om,\psi)$.

(b) $\Rightarrow$ (a).
Similar to theorem~\ref{t:region}, use the special pair in \cref{e:haohao},
\[ \begin{split}
  \llangle S(X); \om, \psi \rrangle
  = \int_{0}^\infty \rho\left(\om f_*^t X^s_{-t} \right) dt
  - \int_{-\infty}^0 \rho\left(\om f_*^t X^u_{-t} \right) dt
  - \rho\left(\eta \psi \right) \\
  = \int_{0}^\infty \rho\left(X^{s}(f^{*t} \om_t) \right) dt
  - \int_{-\infty}^0 \rho\left(X^u(f^{*t} \om_t)  \right) dt
  - \rho\left(X \psi\eps^c \right)
  = \rho(S(\om,\psi)X).
\end{split} \]

(a) $\Rightarrow$ (b).  Same as discrete-time.

The Holder condition of $\nu^{su}$ follows from appendix~\ref{a:holder};  $\nu^c = -\psi \eps^c$ is Holder, so $\nu$ is also Holder.
\end{proof}

\subsection{Applications of adjoint shadowing lemma and discussions}
\hfill\vspace{0.1in}
\label{s:appaslf}

By section~\ref{s:waiyi} and characterization (a), $\cS$ adjointly expresses the shadowing contribution, a significant part of the linear response.

By (c), to numerically compute $\cS(\omega,\psi)$, solve the nonintrusive adjoint shadowing problem in continuous-time,
\[ \begin{split} 
  \nu =  \nu ' + \ueps a \,,
  \quad  \mbox{s.t. }
  \ip{\nu_0, \ueps_0}=0,
\end{split}\]
where $\ueps$ is $u$-many homogeneous adjoint solutions forming a basis of $V^{*u}$, and $\nu'$ is an inhomogeneous solution such that $\nu'_T(F_T) = \psi_T$.
This problem mimics characterization (c).
There are other ways to mimic, for example we previously used a least-squares version of this problem.
The nonintrusive adjoint shadowing algorithm was used on fluid examples with $u=8$, and $M\approx3\times10^6$; 
we only computed $SC$, not the full linear response;
the cost was on the same order of simulating the flow.
The result is not accurate since we did not compute UC, but the result approximately reflects the relation between the averaged observable $\rho(\Phi)$ and the parameter $\gamma$, so only computing SC could sometimes gives something useful for engineering purposes \cite{Ni_nilsas}.

By (b), the adjoint shadowing lemma is also important for computing $UC$ on an orbit, where
\[ \begin{split}
  UC = \lim_{W\rightarrow\infty} - \rho(\phi \frac{\delta^{cu} \tL^{cu}\sigma} \sigma),
  \quad\textnormal{where}\quad 
  \phi:=\int_{-W}^W \Phi\circ f^t -\rho(\Phi) dt .
\end{split} \]
Here $\sigma$ is the conditional measure on center-unstable manifolds, and $\delta^{cu} \tL^{cu}\sigma$ is the center-unstable perturbation of the transfer operator on $\sigma$.
In a forthcoming paper, we shall show that if $X=\delta F$, 
$\eta = \eps^c X$ (so $\eta$ is differentiable along $F$),
then
\[ \begin{split}
 - \frac{\delta^{cu} \tL^{cu}\sigma} \sigma =  \div^{cu}_\sigma X^{cu} = \cS(\div^v \nabla F, \div^v F) X + \div^v X + F(\eta),
\end{split} \]
Note that here $\div^v \nabla F$ is no longer the differential of $\div^v F$, but they still satisfy the constraint in $\cA$.
This justifies our seemingly unnecessary involvement of $\psi$ in the continuous-time shadowing theory.

A `discrete-adjoint’ nonintrusive algorithm for computing $SC$ of continuous-time linear response \cite{Blonigan_2017_adjoint_NILSS} was given before our nonintrusive adjoint shadowing algorithm \cite{Ni_nilsas}. 
The discrete adjoint is obtained by transposing the linear system in our tangent nonintrusive shadowing algorithm \cite{Ni_NILSS_JCP}.
However, its cost is twice larger than ours, since it requires computing and saving and reloading all unstable tangent solutions, which are forward-propagating vectors, at checkpoints.
Moreover, its theory was missing or less satisfactory.
For example, its solution was not continuous; it did not define a dual operator; the boundedness and well-definedness (as a covector field) of its solution were unclear; its description is more complicated than ours and less similar to the conventional backpropagation method; it can not give a $SC+UC$ decomposition for the linear response of finitely-long time system; finally, it can not be used in the unstable contribution of the linear response.

\section{A future direction: adding randomness}
\label{s:future}

The main drawback of our current works, including shadowing and equivariant divergence formulas, is that they involve the Jacobian matrices. 
When the Jacobian matrix has bad properties, for example non-hyperbolicity, both formulas are affected, although the algorithm for $SC$ is more robust than $UC$.
$ UC$ is more fragile since the equivariant divergence formula has the term $\cS(\div^v \nabla f_*)$, and both $\cS$ and $\div^v \nabla f_*$ are affected by non-hyperbolicity.
For non-hyperbolic problems such as Lorenz 96, we can still compute $SC$ in practice, but not $UC$.

The main thing to do next should be to overcome the non-hyperbolicity in general.
On one end, it is known that some non-hyperbolic systems have no linear response, so we can not expect to compute it accurately \cite{Baladi2007,Wormell2019}.
However, looking at the observable-parameter plot of such systems, it seems that there is still some trend between the $\gamma$ and $\rho(\Phi)$ \cite{np}. 
So, for practical purposes, we may still ask how to compute an approximate linear response which reflects this trend, but also with the least approximation error.
Such an approximation would still be very useful for practical purposes such as computational design optimization.

It seems that a plausible idea is to add random noise when the hyperbolicity is bad.
This has two benefits, the first is that the linear response of the random dynamical system has another formula, the kernel differentiation formula (also called the likelihood ratio method by probabilists).
This formula differentiates only the kernel but not the Jacobian matrix, so the bad properties of the Jacobian, including non-hyperbolicity, does not affect the algorithm at all \cite{Rubinstein1989,Reiman1989,Glynn1990,Hairer2010}.
This formula admits an `ergodic theorem' version, which runs on only one orbit, and the size of the integrand does not grow with the orbit length \cite{np}.
Another benefit is that the random noise helps to quickly move out of non-hyperbolic regions.
In particular for singular hyperbolic flows, non-hyperbolicity is typically associated with a stagnant point, and it takes a long to move out of a neighborhood (which is actually the cause of the non-hyperbolicity).

We should also try to add minimal amount of noise to reduce the associated error.
However, the kernel differentiation formula is expensive when the noise is small, so we can not add a small noise everywhere.
So it seems that we have to add noise only to local regions where the deterministic formulas do not work, and we use the kernel formulas only occasionally.
Since the area of `local' regions tends to be small when the dimension is high, it seems that we can have a small noise-induced error by adding a large noise to only a small region.

It seems that there are two major challenges in this random-deterministic switching program.
The first is to extend the pathwise perturbation formula and divergence formulas to random dynamical systems.
The continuous-time case is more difficult since we need to deal with stochastic calculus.
In particular, it requires extra care to recover the time-reparametrization trick which we used for continuous-time shadowing.
The second difficulty is to let the stochastic and deterministic formulas communicate.
The Bismut-Elworthy-X.Li formula is an example where the pathwise perturbation formula communicates information to the kernel differentiation formula under perturbation to initial conditions \cite{Bismut1984,Elworthy1994,Ren2019}, but we need more formulas for other types of communications for parameter perturbations, which might be more difficult.

Moreover, there should be other problem-specific ways to add noise.
For dynamical systems obtained from PDE, we might attempt to add space-time noise according to features in the physical space.
For example, we might try to add noise only to large vortices.

\section*{Acknowledgements}

The author thanks Yao Tong and Chaitanya Talnikar for helpful discussions.

\begin{appendix}

\section{Holder continuity of `split-propagate' scheme}
\label{a:holder}

The `split-propagate' scheme, as used in the expansion formula \cref{e:expand_v,e:haohao}, and characterization (b) in adjoint shadowing lemmas, is common in the linear response theory.
It was proved by a functional approach that the shadowing vector $v$ in \cref{e:expand_v} is Holder continuous, using the fact that $v$ is the derivative of a conjugation map \cite{Ruelle_diff_maps}.
But that proof does not seem to work directly on continuous-time or adjoint cases.
We give a more basic proof which also works on continuous-time or adjoint cases.

\begin{lemma} \label{l:holder}
  If $X$ is Holder on $K$, then $v:=S(X)$ in~\cref{e:expand_v} is Holder on $K$.
\end{lemma}

\begin{proof}
Pick an finite open cover of $K$ by coordinate charts, denoted by $\{U_i\}$.
Let $\delta$ be the Lebesgue number of this open cover, and $d$ be the distance function.
Take any $x, y$ with $d(x,y) \le \delta$, there is a chart containing both $x$ and $y$, the goal is to show that $v$ is Holder under this chart.

Let $\mu:=\max \{|f_*|, |f_*^{-1}|\}$, where $|\cdot|$ is the operator norm.
Let $N(x,y)$ be the number such that 
\[ \begin{split}
  \mu^N d(x,y) < \delta \le \mu^{N+1} d(x,y).
\end{split} \]
So for each $-N\le k\le N$, $x_k:=f^k x$ and $y_k$ are in the same chart, denoted by $U_k$.
Partition the sequence in the expansion of $v$,
\[ \begin{split}
  v
  = \sum_{0\le n\le N} f^{n}_{*} P^s X_{-n} 
  + \sum_{n\ge N+1} f^{n}_{*} P^s X_{-n} 
  - \sum_{-N\le n\le -1 } f^{n}_{*} P^u X_{-n} 
  - \sum_{n\le -N-1} f^{n}_{*} P^u X_{-n} .
\end{split} \]
We only look at terms with $n\ge0$; the $n\le -1$ part is similar.

Take $v(x)-v(y)$ under the chart containing $x$ and $y$,
\[ \begin{split}
  v(x) - v(y)
  = \sum_{0\le n\le N} 
  f^{n}_{*} P^s X_{-n} (x)
  - f^{n}_{*} P^s X_{-n} (y)\\
  + \sum_{n\ge N+1} f^{n}_{*} P^s X_{-n} (x)
  - f^{n}_{*} P^s X_{-n} (y)
  - \sum_{n\le 0 } \cdots.
\end{split} \]
Bound the terms in the second sum by triangle inequality, so
\[ \begin{split}
  |f^{n}_{*} P^s X_{-n} (x) - f^{n}_{*} P^s X_{-n} (y)|
  \le C \lambda^n
\end{split} \]
where $C$ depends on the hyperbolicity constant, the sup operator norm of $P^s$, and the sup of $X$, all taken on the entire $K$.

For the $n$th term, $n\le N$, for any $0\le k\le n$, we can write $f_s(x_{-k})$ and $f_s(y_{-k})$ as matrices under the same chart $U_{-k}$, where $f_s := f_*P^s$.
We can also write $X(x_{-n})$ and $X(y_{-n})$ as Euclidean vectors in chart $U_n$.
Apply the finite-difference version of Leibniz rule, 
\[ \begin{split}
  f^{n}_{*} P^s X_{-n} (x)
  - f^{n}_{*} P^s X_{-n} (y)
  = f^{n}_{s} X (x_{-n})
  - f^{n}_{s} X (y_{-n}) \\
  = 
  f^{n}_{s} (x_{-n}) (X(x_{-n}) - X(y_{-n}))
  + \sum_{1\le k\le n}
  f^{k-1}_{s} (x_{1-k}) (f_{s}(x_{-k})-f_s(y_{-k})) f^{n-k}_{s} X(y_{-n}).
\end{split} \]
Since $f_s$ is a Holder continuous operator on $K$, there are two constants $C$ and $\alpha$ on all $U_i$'s, such that $|f_{s}(x)-f_s(y)|\le C d^\alpha(x,y)$ for any $x, y$ in the same chart.
Let $\alpha, C$ also be the Holder constants for $X$, and let $\lambda \mu^\alpha\neq 1$.
Note that
\[ \begin{split}
  |f_s^k| 
  = |(f_* P^s)^k|
  = |f_*^k P^s| \le C\lambda ^k,
\end{split} \]
so
\[ \begin{split}
  |f^{n}_{s} X_{-n} (x)
  - f^{n}_{s} X_{-n} (y)|
  \le C \lambda^{n} d^\alpha(x_{-n},y_{-n}) 
  + \sum_{1\le k\le n} C \lambda^{k-1} d^\alpha(x_{-k},y_{-k}) \lambda^{n-k} \\
  \le C \lambda^{n} \sum_{1\le k\le n}  (\mu^k d(x,y))^\alpha 
  \le C \lambda^n \mu^{n\alpha} d^\alpha(x,y). 
\end{split} \]
Here $C$ changes from line to line, but each $C$ is uniform for all $x, y$ in $K$ if $d(x,y)\le \delta$.

Summarizing, 
\[ \begin{split}
|v(x)-v(y)|
  \le \sum_{0\le n\le N} C \lambda^n \mu^{n\alpha} d^\alpha(x,y)
  + \sum_{n\ge N+1} C \lambda^n\\
  \le 
  C d^\alpha(x,y)
  + C \lambda^N \mu^{N\alpha} d^\alpha(x,y)
  + C \lambda^N .
\end{split} \]
Since $ \mu^N d(x,y) < \delta \le \mu^{N+1} d(x,y) $,
\[ \begin{split}
|v(x)-v(y)|
  \le C d^\alpha(x,y) + C \lambda^N \delta^\alpha + C \lambda^N
  \le C d^\alpha(x,y) + C \lambda^N.  
\end{split} \]
Since $\lambda < 1 < \mu$, there is $\beta>0$ such that $\lambda = \mu^{-\beta}$, so
\[ \begin{split}
  \lambda^N = \mu^{-\beta N} \le C d^\beta(x,y).
\end{split} \]
Finally, 
\[ \begin{split}
|v(x)-v(y)|
  \le C d^\alpha(x,y) + C d^\beta(x,y),
\end{split} \]
so $v$ is Holder on $K$.
\end{proof}

As we can see, this proof works on a somewhat more general formula, which not only applies to shadowing formulas, but also applies to the tangent version of the equivariant divergence formula in \cite{fr}.
We say that a map $g:T \cM\rightarrow T \cM$ covers $f$ on $K$, if for all $x\in K$, $g_x$ is a map from $T_x\cM$ to $T_{fx}\cM$; we sometimes omit $x$ and write only $g$.
We say $g$ is Holder if it is holder under all coordinate charts.
We say $g$ exponentially decays if there is $0<\lambda<1$ and $C>0$, such that $|g^n|<C\lambda^n$.
For example, $f_s=f_*P^s$ in the proof above is exponentially decaying and Holder continuous.

\begin{lemma} [Holder for `decay-sum' scheme] \label{l:www} 
If $f$ is a bijection map on a compact set $K$, $f$ and $f^{-1}$ are Lipschitz continuous; $g$ covers $f$, $h$ covers $f^{-1}$, $g$ and $h$ are Holder continuous and exponentially decaying; $Y$ is Holder continuous.
Then the following vector fields on $K$ are Holder continuous:
\[ \begin{split}
  S_1(x) := \sum_{n\ge 0} g^n Y(f^{-n}x),
  \quad \textnormal{} \quad 
  S_2(x) := \sum_{n\ge 0} h^n Y(f^{n}x).
\end{split} \]
\end{lemma}

To prove this, just replace $f_*P^s$ by $g$, $X$ by $Y$, $v$ by $S_1$ in the proof above.
For $S_2$, repeat the proof with $f_*^{-1}P^u$ replaced by $h$.

\end{appendix}

\bibliographystyle{abbrv}
{\footnotesize\bibliography{library}}

\end{document}